\documentclass[11pt]{article}

\usepackage[a4paper,margin=1in]{geometry}
\usepackage{amsmath,amssymb,amsthm,mathtools}
\usepackage{enumitem}
\usepackage{graphicx}
\usepackage[colorlinks=true,allcolors=blue]{hyperref}
\usepackage{todonotes}

\newtheorem{theorem}{Theorem}[section]
\newtheorem{lemma}[theorem]{Lemma}

\newtheorem{corollary}[theorem]{Corollary}

\newcommand{\Prob}{\mathbb{P}}
\newcommand{\E}{\mathbb{E}}
\newcommand{\dist}{\operatorname{dist}}

\title{Embedding dense minors via crux}
\author{
Xia Wang\thanks{School of Mathematics, Shandong University, Jinan, China and Extremal Combinatorics and Probability Group (ECOPRO), Institute for Basic Science (IBS), Daejeon, South Korea. Supported by China Scholarship Council and  the Institute for Basic Science (IBS-R029-C4). Email: {\tt xiawang@mail.sdu.edu.cn}.}
\and
Donglei Yang\thanks{School of Mathematics, Shandong University, Jinan, China. Email: {\tt dlyang@sdu.edu.cn}. Supported by Natural Science Foundation of China (12571374).}
}
\date{}

\begin{document}
\maketitle

\begin{abstract} Let $G$ be a graph with average degree $d>0$ and $C_\alpha(G)$ be the minimum order of a subgraph whose average degree is at least $\alpha d$ for any fixed $0<\alpha<1$. Then $G$ has a minor of average degree $\Omega(\sqrt{d\,C_\alpha(G)})$, which is asymptotically tight up to a constant factor.
This reveals a unified mechanism underlying previous results on dense minors in locally sparse graphs.
\end{abstract}

\section{Introduction}
We say that a graph $G$ contains an \emph{$H$-minor}
if there exists a family $\{T_v:v\in V(H)\}$ of pairwise disjoint nonempty subsets of $V(G)$ such that $G[T_v]$ is connected for every $v\in V(H)$, and $e_G(T_u,T_v)>0$ whenever $uv\in E(H)$.
The sets $T_v$ are called the \emph{branch sets}
of the minor. When forming a minor, loops are
deleted and parallel edges are replaced by
single edges.

The study of graph minors is a central topic in structural and
extremal graph theory. In 1943, Hadwiger~\cite{Hadwiger1943}
conjectured that every graph with chromatic number at least $t$
contains a $K_t$-minor. The cases $t\le6$ follow from classical
results, including the reductions of Wagner~\cite{Wagner1937} and that of
Robertson, Seymour and Thomas~\cite{RobertsonSeymourThomas1993}
to the Four-Color Theorem. 
An important extremal counterpart asks how large
the average degree of a graph must be to force
a $K_t$-minor. 
For sufficiently large $n$, the sharp extremal
bounds for $K_t$-minor-free graphs are known
for several small values of $t$.
For $t=6,7$, Mader~\cite{Mader} proved that
every $n$-vertex graph without a $K_t$-minor
has at most $(t-2)n-\binom{t-1}{2}$ edges.
For $t=8$, J{\o}rgensen~\cite{Jorgensen}
established the sharp upper bound $6n-20$,
with equality only for graphs obtained
by gluing copies of $K_{2,2,2,2,2}$
along $K_5$.
For $t=9$, Song and Thomas~\cite{SongThomas}
proved the sharp upper bound $7n-27$,
with equality only for $K_{2,2,2,3,3}$
and graphs obtained by gluing copies of
$K_{1,2,2,2,2,2}$ along $K_6$.
For large $t$, Kostochka~\cite{Kostochka}
and Thomason~\cite{Thomason} established
that the sharp order of this threshold is
$t\sqrt{\log t}$.
In particular, every graph without a $K_t$-minor is
$O(t\sqrt{\log t})$-colorable.

Substantial progress on the coloring problem has been made in
recent years. Norin, Postle and Song~\cite{NorinPostleSong2023}
proved that $K_t$-minor-free graphs are
$O(t(\log t)^\beta)$-colorable for every fixed $\beta>1/4$.
Delcourt and Postle~\cite{DelcourtPostle2025} improved this to $O(t\log\log t)$ colors and reduced the linear Hadwiger conjecture to the case of graphs of relatively small order.
Recently, Liu and Luo~\cite{LiuLuo2026} obtained the bound $O(t\log\log\log t)$. Norin and
Steiner~\cite{NorinSteiner2026} presented an AI-assisted proof of the linear Hadwiger conjecture, showing that every $K_t$-minor-free graph is $Ct$-colorable for some absolute constant $C$. Wood~\cite{Wood2026} subsequently improved
this constant using AI-assisted optimization. 
An OpenAI-generated preprint~\cite{OpenAIHadwiger2026} proposed counterexamples to Hadwiger's original conjecture, even among graphs with independence number at most two. 

A related extremal question asks for minors of large average degree, without requiring them to be complete. For instance, Hendrey, Norin, Steiner and Turcotte~\cite{HendreyNorinSteinerTurcotte2025} proved that for sufficiently large $t$, every graph of average degree at least $t-1$ contains a minor on $t$ vertices and at least $(\sqrt2-1-o(1))\binom t2$ edges.
In a different direction, Norin and Seymour~\cite{NorinSeymour2026}
showed that every graph with independence number at most two has a minor whose order equals the chromatic number of the graph and whose edge density is at least $0.98688-o(1)$.
These results illustrate that dense minors are natural objects even when clique minors of the same order are not guaranteed.

For a graph $G$ with $d(G)>0$ and
$0<\alpha<1$, let \emph{$\alpha$-crux} $C_\alpha(G)$
be the minimum order of a subgraph $H\subseteq G$ satisfying $d(H)\ge\alpha d(G)$. This parameter was first considered by Liu and Montgomery~\cite{LiuMontgomery2017} in their study of clique subdivisions and was later named the \emph{crux} by Haslegrave, Hu, Kim, Liu, Luan and Wang~\cite{HaslegraveEtAl}. Haslegrave et al.~\cite{HaslegraveEtAl} established results on long cycles in terms of the crux. Im, Kim, Kim and Liu~\cite{ImKimKimLiu}
subsequently obtained bounds for clique subdivisions in terms of the crux.
More recently, Yang and Yang~\cite{YangYang}
confirmed a conjecture of Liu and Montgomery,
showing that, for a sufficiently small fixed
$\alpha>0$, every graph $G$ contains a clique
subdivision of order $\Omega\!\left(\min\{d(G),
\sqrt{C_\alpha(G)/\log C_\alpha(G)}\}\right)$.

Complete graphs show that large average degree alone
does not force a minor of substantially larger average
degree. In contrast, stronger bounds are known for
graphs with forbidden bipartite subgraphs.
Krivelevich and Sudakov~\cite[Theorems~4.5 and~4.7]{KrivelevichSudakov2009} proved that, for fixed integers $2\le s\le t$, every $K_{s,t}$-free graph of average degree $d$ contains a minor of average degree
$\Omega_{s,t}(d^{1+1/(2(s-1))})$, while every $C_{2k}$-free graph with fixed $k\ge2$ contains a minor of average degree $\Omega_k(d^{(k+1)/2})$.
Norin, Postle and Song~\cite[Theorem~5.3]{NorinPostleSong2023} extended such bounds to general forbidden bipartite graphs in terms of their Tur\'an exponents.

A common feature of these examples is that any subgraph retaining a fixed proportion of the average degree must have many vertices. The crux provides a quantitative measure of this phenomenon and naturally leads to the question of whether every graph $G$ of average degree $d>0$ contains a minor of average degree $\Omega(\sqrt{d\,C_\alpha(G)})$ for each fixed $0<\alpha<1$. This question was also raised by Hong Liu in private discussions. Our main theorem answers it affirmatively.

\begin{theorem}\label{thm:crux}
For every $0<\alpha<1$, there exists
$c_\alpha>0$ such that every graph $G$ with
$d(G)>0$ contains a $J$-minor with
\[
d(J)\ge c_\alpha\sqrt{d(G)C_\alpha(G)}.
\]
Moreover, $c_\alpha$ can be chosen to be the same
positive constant for all $0<\alpha\le1/2$.
\end{theorem}

The bound in Theorem~\ref{thm:crux} is asymptotically tight up to a constant factor, as witnessed by complete graphs. Indeed, for $G=K_n$, we have $d(G)=n-1$ and $C_\alpha(G)=\Theta_\alpha(n)$, while every minor of $G$ has average degree at most $n-1$.
Moreover, Norin, Postle and Song~\cite[Theorem~2.4]{NorinPostleSong2023} proved that, for every fixed $\delta>0$, a graph of sufficiently large average degree $d$ either contains a minor of average degree at least $D>d$, or has a subgraph $H$ with $|H|=O_\delta((D/d)^\delta D^2/d)$ and $d(H)=\Omega_\delta((D/d)^{-\delta}d)$. Our Theorem~\ref{thm:crux} strengthens this result by showing that, for every $d>0$ and $D>d$, either such a minor with average degree $D$ exists or there is a subgraph $H$ with $|H|=O(D^2/d)$ and $d(H)\ge d/2$, where the implied constant is absolute.

In Section~\ref{sec:applications}, we recover
the forbidden subgraph bounds mentioned above
and present further applications to subgraphs
of Hamming graphs and to pseudorandom graphs.

\paragraph{Organization of the paper.}
The rest of the paper is organized as follows.
In Section~\ref{sec:prelim}, we introduce the notation
and definitions used throughout the paper.
In Section~\ref{sec:crux}, we prove Theorem~\ref{thm:crux}.
In Section~\ref{sec:applications}, we present applications
to graphs with forbidden subgraphs, subgraphs of Hamming
graphs, and pseudorandom graphs.

\section{Notation}\label{sec:prelim}

Let $G$ be a finite simple undirected graph with
vertex set $V(G)$ and edge set $E(G)$.
Write $|G|:=|V(G)|$ and $e(G):=|E(G)|$.
The \emph{average degree} of $G$ is
$d(G):=2e(G)/|G|$ when $|G|\ge1$.
Let $\delta(G)$ and $\Delta(G)$ denote the minimum
and maximum degrees of $G$, respectively, and $d_G(v)$ denote the degree of a vertex $v$.
For a subset $X\subseteq V(G)$, let $G[X]$
denote the induced subgraph of $G$ on $X$,
and $G-X:=G[V(G)\setminus X]$.
For any two disjoint subsets $X,Y\subseteq V(G)$,
let $G[X,Y]$ denote the bipartite subgraph
of $G$ with parts $X$ and $Y$, and $e_G(X,Y)$ denote the number of edges between $X$ and $Y$.
The \emph{external neighborhood} of $X$ is
\[
N_G(X):=\{v\in V(G)\setminus X:
uv\in E(G)\text{ for some }u\in X\}.
\]
We write $A_1\mathbin{\dot\cup}\cdots
\mathbin{\dot\cup}A_k$ for the union of
pairwise disjoint sets $A_1,\ldots,A_k$.
All logarithms are natural, and $[m]:=\{1,\ldots,m\}$
for every positive integer $m$.
For a graph $G$ with $|G|\ge2$, the
\emph{Cheeger constant} of $G$ is defined as
\[
h(G):=
\min_{\substack{\varnothing\ne X\subseteq V(G)\\
|X|\le |G|/2}}
\frac{e_G(X,V(G)\setminus X)}{|X|}.
\]
This is the edge expansion parameter used in
\cite{KrivelevichNenadov,LiXuZhou},
which we use throughout the proof of Theorem~\ref{thm:crux}.

\section{Proof of Theorem \ref{thm:crux}}
\label{sec:crux}

We first establish the auxiliary lemmas needed for
Theorem~\ref{thm:crux} and then complete its proof.
The following lemma is a quantitative variant of
\cite[Lemma~4.1(ii)]{LiXuZhou}. While their result leaves
the dependence on the expansion parameter $\varepsilon$
implicit, we obtain a uniform bound with polynomial
dependence on $\varepsilon$. This allows the constant in
Theorem~\ref{thm:crux} to depend only on $\alpha$.
The proof follows their approach and is given in
Appendix~\ref{app:crux-cheeger}.

\begin{lemma}\label{lem:crux-cheeger}
There exists an absolute constant $a>0$ such that
the following holds.
Let $G$ be an $n$-vertex graph with $n\ge2$,
let $D\ge1$ be an integer, and let
$0<\varepsilon\le1$.
If $\Delta(G)\le D$ and $h(G)\ge\varepsilon D$,
then $G$ contains a $J$-minor satisfying $d(J)\ge a\varepsilon^8\sqrt{nD}$.
\end{lemma}

The next crucial lemma finds an induced subgraph with controlled
maximum degree and Cheeger constant while retaining
average degree comparable to that of the original graph, which is of independent interest.

\begin{lemma}\label{lem:crux-extract}
Let $0<\alpha<1$ and
$\tau:=\min\{1/4,-(\log\alpha)/2\}$.
For every $n$-vertex graph $G$ with $d:=d(G)>0$,
there exists an induced subgraph $H_1$ with
$h_1:=|H_1|\ge C_\alpha(G)$ and
$d':=d(H_1)\ge\mathrm e^{-\tau}d$ satisfying the
following properties, where
$L:=\log(\mathrm e+h_1/C_\alpha(G))$
and $\eta:=\tau/(32(\mathrm e+1)L^2)$.

\begin{enumerate}[label=\rm(\roman*)]
\item
For every partition
$V(H_1)=X\mathbin{\dot\cup}Y\mathbin{\dot\cup}Z$
with $0<|X|\le3h_1/4$ and $|Y|\le\eta|X|$,
we have $e_{H_1}(X,Z)\ge(3/2)\eta d'|X|$.

\item
There exists an induced subgraph $H_2$ of $H_1$
with $h_2:=|H_2|>(3/4-\eta/16)h_1\ge h_1/2$,
$\Delta(H_2)\le16d'/\eta$, and
$h(H_2)\ge\eta d'$.
\end{enumerate}

Moreover, $d'\ge1$ and $h_2\ge2$.
\end{lemma}

\begin{proof}
For $x\ge0$, define \[\lambda(x):=\tau\left(1-\frac{1}{\log(e+\frac{x}{C_\alpha(G)})}\right).\]
The function $\lambda$ is increasing, takes values
in $[0,\tau)$, and $\lambda'(x)=\tau/(( e C_\alpha(G)+x)\log^2(e+x/C_\alpha(G)))$.
Choose an induced subgraph $H_1$ of $G$ maximizing
$d(H_1)\mathrm e^{-\lambda(|H_1|)}$.
Comparing $H_1$ with $G$, we have $d'\ge d e^{\lambda(h_1)-\lambda(n)} \ge e^{-\tau}d>\alpha d$ as $0\le \lambda (h_1)\le \lambda(n)<\tau$, and $h_1\ge C_\alpha(G)\ge2$.
By maximality, we have
\begin{equation}\label{eq:crux-maximality}
d(Q)\le d' e^{\lambda(|Q|)-\lambda(h_1)}~ \text{for every nonempty induced subgraph $Q\subseteq H_1$}.
\end{equation}

We first prove (i).
Consider a partition as in the statement. Let $s:=|X|$, $y:=|Y|$, $p:=h_1-s$, and $q:=s+y$.
Since $s\le 3h_1/4$, $p\ge h_1/4$, 
$e(H_1)=e(H_1[X\cup Y])+e(H_1[Y\cup Z])
-e(H_1[Y])+e_{H_1}(X,Z)$, $h_1-q=p-y$
and applying~\eqref{eq:crux-maximality} to $H_1[X\cup Y]$ and $H_1[Y\cup Z]$, we have
\begin{align}
2e_{H_1}(X,Z)
\ge d'\left(
h_1-q e^{\lambda(q)-\lambda(h_1)}
-p e^{\lambda(p)-\lambda(h_1)}
\right)\ge d'\left(
p\bigl(1- e^{\lambda(p)-\lambda(h_1)}\bigr)
-y\right).
\label{eq:crux-cut}
\end{align}
Let $u:=\lambda(h_1)-\lambda(p)$.
As $0\le u\le\tau\le1/4$, $1- e^{-u}\ge u/2$.
Since $\lambda'$ is decreasing, we have
\[
u=\lambda(h_1)-\lambda(p)
=\int_p^{h_1}\lambda'(x)\,dx
\ge s\lambda'(h_1)
=\frac{\tau s}{( e C_\alpha(G)+h_1)L^2}.
\]
As $p\ge h_1/4$, $h_1\ge C_\alpha(G)$ and $y\le \eta s$, by \eqref{eq:crux-cut}, we have
\[2e_{H_1}(X,Z)\ge d'\left(p(1- e^{-u})-y\right)\ge d'\left(\frac{\tau h_1s}{8( e C_\alpha(G)+h_1)L^2}-y\right)
\ge d'\left(\frac{\tau s}{8( e+1)L^2}-y\right)
\ge 3\eta d' s,\]
which proves (i).
Taking $X=\{v\}$ and $Y=\varnothing$, we see
that every vertex of $H_1$ has positive degree. Hence $d'\ge1$.

We next prove (ii).
Let $B=\{v\in V(H_1): d_{H_1}(v)>16d'/\eta\}$.
Since $\sum_{v\in B} d_{H_1}(v)\le\sum_{v\in V(H_1)} d_{H_1}(v) =d'h_1$, we have $|B|<\eta h_1/16$.
Let $R_0:=V(H_1)\setminus B$ and $W_0:=\varnothing$. For each $i\ge1$, if there exists a set $S_i\subseteq R_{i-1}$ such that $0<|S_i|\le |R_{i-1}|/2$ and $e_{H_1}(S_i,R_{i-1}\setminus S_i) <\eta d'|S_i|$, then set $R_i:=R_{i-1}\setminus S_i$ and
$W_i:=W_{i-1}\cup S_i$. Otherwise, terminate
the process.
Since $S_i$ is nonempty at each step, we have
$|R_i|<|R_{i-1}|$. Thus, as $H_1$ is finite,
the process terminates after finitely many steps.
For every $i\ge1$, since
$R_i\subseteq R_j$ for all $j\le i$, we have
\begin{align*}
e_{H_1}(W_i,R_i)
&=\sum_{j=1}^i e_{H_1}(S_j,R_i)\le\sum_{j=1}^i e_{H_1}(S_j,R_j)<\eta d'\sum_{j=1}^i|S_j|
=\eta d'|W_i|.
\end{align*}
We claim that $|W_i|<h_1/4$ for every $i\ge 0$. Suppose not, and let $i$
be the smallest index such that
$|W_i|\ge h_1/4$. Then
$|W_{i-1}|<h_1/4$ and
$|S_i|\le|R_{i-1}|/2$. Since
$|R_{i-1}|=h_1-|B|-|W_{i-1}|$, we have
\[
\frac{h_1}{4}\le |W_i|
\le\frac{h_1-|B|+|W_{i-1}|}{2}
<\frac{5h_1}{8}<\frac{3h_1}{4}.
\]
Moreover, $|B|\le\eta h_1/16\le\eta|W_i|$.
Thus the partition
$V(H_1)=W_i\mathbin{\dot\cup}B
\mathbin{\dot\cup}R_i$
satisfies the hypotheses of (i), which implies
$e_{H_1}(W_i,R_i)\ge(3/2)\eta d'|W_i|$,
a contradiction.

Let $R_t$ be the remaining set when the
process terminates, and $H_2:=H_1[R_t]$.
Since $|W_t|<h_1/4$ and
$|B|\le\eta h_1/16$, we have
\[
h_2=|R_t|=h_1-|B|-|W_t|
>\left(\frac34-\frac{\eta}{16}\right)h_1
\ge\frac{h_1}{2}.
\]
Furthermore, by the definition of $B$,
$\Delta(H_2)\le16d'/\eta$.
By the stopping condition, every $S\subseteq R_t$ with $0<|S|\le|R_t|/2$
satisfies
$e_{H_1}(S,R_t\setminus S)\ge\eta d'|S|$.
Since $H_2=H_1[R_t]$, this implies
$h(H_2)\ge\eta d'$.
Finally, $h_1\ge2$ and $h_2>h_1/2$
imply $h_2\ge2$, completing the proof.
\end{proof}

Now, we are ready to give a short proof of Theorem~\ref{thm:crux}.

\begin{proof}[Proof of Theorem~\ref{thm:crux}]
Fix $0<\alpha<1$. Let $\tau:=\min\{1/4,-(\log\alpha)/2\}$.
By Lemma~\ref{lem:crux-extract}, there exist
induced subgraphs $H_2\subseteq H_1\subseteq G$
satisfying the conclusions $(i)$ and $(ii)$,
with parameters $h_1,h_2,d',L$, and $\eta$.

Let $D:=\lceil16d'/\eta\rceil$ and
$\varepsilon:=\eta^2/32$.
Since $d'\ge1$ and $\eta<1$, we have
$D\le32d'/\eta$. Thus
$\Delta(H_2)\le D$ and
$h(H_2)\ge\eta d'\ge\varepsilon D$.
Applying Lemma~\ref{lem:crux-cheeger} to $H_2$,
we obtain a $J$-minor in $G$ with
$d(J)\ge a\varepsilon^8\sqrt{h_2D}$,
where $a>0$ is an absolute constant.

Since $h_2\ge h_1/2$ and
$D\ge d'\ge\mathrm e^{-\tau}d\ge d/2$,
we have $\sqrt{h_2D}\ge\frac12\sqrt{h_1d}$.
Moreover, by the definition of $\eta$,
we have
$\varepsilon=\tau^2/
(32^3(\mathrm e+1)^2L^4)$.
Consequently,
\[
\begin{aligned}
d(J)
&\ge a\varepsilon^8\sqrt{h_2D}\ge\frac{a\tau^{16}}
{2\cdot32^{24}(\mathrm e+1)^{16}}
\frac{\sqrt{h_1d}}
{\log^{32}(\mathrm e+h_1/C_\alpha(G))}\\
&=\frac{a\tau^{16}}
{2\cdot32^{24}(\mathrm e+1)^{16}}
\frac{\sqrt{x}}
{\log^{32}(\mathrm e+x)}
\sqrt{dC_\alpha(G)},
\end{aligned}
\]
where $x:=h_1/C_\alpha(G)\ge1$.
Since
$\log^{32}(\mathrm e+x)
\le64^{32}\sqrt{(\mathrm e+1)x}$,
we obtain
$d(J)\ge c_\alpha\sqrt{dC_\alpha(G)}$, where
\[
c_\alpha:=
\frac{a\tau^{16}}
{2\cdot32^{24}\cdot64^{32}
(\mathrm e+1)^{33/2}}>0.
\]
Since $\tau=1/4$ for all $0<\alpha\le1/2$,
the above choice of $c_\alpha$ is the same
positive constant throughout this range.
\end{proof}

\section{Applications}\label{sec:applications}

We give several applications of Theorem~\ref{thm:crux}
to graphs with forbidden subgraphs, subgraphs of Hamming
graphs, and pseudorandom graphs.
We first introduce some tools.

\begin{lemma}[\cite{Kostochka,Thomason}]
\label{lem:applications-KT}
There exists an absolute constant $\lambda_{\mathrm{KT}}>0$
such that every graph of average degree $d\ge2$
contains a $K_r$-minor for some
$r\ge \lambda_{\mathrm{KT}}d/\sqrt{\log d}$.
\end{lemma}

We also use two classical extremal estimates.
For a graph $F$, let $\operatorname{ex}(n,F)$ denote
the maximum number of edges in an $n$-vertex $F$-free graph.

\begin{lemma}[\cite{KovariSosTuran1954}]
\label{lem:applications-KST}
For every fixed pair of integers $2\le s\le t$, $\operatorname{ex}(n,K_{s,t})
=O_{s,t}(n^{2-1/s})$.
\end{lemma}

\begin{lemma}[\cite{BondySimonovits1974}]
\label{lem:applications-even-cycles}
For every fixed integer $k\ge2$, $\operatorname{ex}(n,C_{2k})=O_k(n^{1+1/k})$.
\end{lemma}

For integers $m\ge1$ and $q\ge2$, let $H(m,q)$ be
the Hamming graph on $[q]^m$, in which two vertices
are adjacent if and only if they differ in exactly
one coordinate. In particular, $H(m,2)$ is the
$m$-dimensional hypercube $Q_m$.
We shall use the following result of Squier, Torrence
and Vogt~\cite{SquierTorrenceVogt2001}.

\begin{lemma}[\cite{SquierTorrenceVogt2001}]\label{lem:applications-Hamming-isoperimetry}
Every nonempty subgraph $F$ of $H(m,q)$ satisfies
$2e(F)\le (q-1)|F|\log_q|F|$.
Consequently, $|F|\ge q^{d(F)/(q-1)}$.
\end{lemma}

\subsection{Graphs with forbidden subgraphs}

Our first application shows that an extremal number
estimate yields a corresponding dense minor bound.

\begin{corollary}\label{cor:forbidden}
Let $F$ be a fixed bipartite graph satisfying
$\operatorname{ex}(n,F)\le Cn^\gamma$
for some constants $C>0$ and $\gamma>1$.
Then every $F$-free graph $G$ with sufficiently
large average degree $d$ contains a $J$-minor such that
\[
d(J)=\Omega_{F,C,\gamma}
\left(d^{\frac{\gamma}{2(\gamma-1)}}\right).
\]
Consequently, $G$ contains a clique minor of order
$\Omega_{F,C,\gamma}
\left(d^{\frac{\gamma}{2(\gamma-1)}}/\sqrt{\log d}\right)$.
\end{corollary}

\begin{proof}
Let $H\subseteq G$ satisfy $d(H)\ge d/2$.
Since $H$ is $F$-free, we have
$d|H|/4\le e(H)\le C|H|^\gamma$.
Thus $|H|\ge(d/(4C))^{1/(\gamma-1)}$, and hence
$C_{1/2}(G)=\Omega_{F,C,\gamma}
(d^{1/(\gamma-1)})$.
Applying Theorem~\ref{thm:crux} with $\alpha=1/2$
gives the stated lower bound on $d(J)$.
The clique minor assertion follows from
Lemma~\ref{lem:applications-KT}.
\end{proof}

\noindent\textbf{Remark.}
The following examples show that Corollary~\ref{cor:forbidden} recovers the polynomial exponents in several known results.

\begin{itemize}
\item
Since $\operatorname{ex}(n,C_4)=O(n^{3/2})$,
every $C_4$-free graph of average degree $d$
contains a $J$-minor with $d(J)=\Omega(d^{3/2})$,
and a clique minor of order
$\Omega(d^{3/2}/\sqrt{\log d})$.
The dense minor bound recovers the case $k=2$
of Theorem~4.7 of Krivelevich and
Sudakov~\cite{KrivelevichSudakov2009},
which extends the earlier girth-five result of
K\"uhn and Osthus~\cite[Corollary~3]{KuhnOsthus2003}.

\item
For fixed $2\le s\le t$,
Lemma~\ref{lem:applications-KST} gives
$\operatorname{ex}(n,K_{s,t})
=O_{s,t}(n^{2-1/s})$.
Hence every $K_{s,t}$-free graph of average degree
$d$ contains a $J$-minor with $d(J)=\Omega_{s,t}(d^{1+\frac{1}{2(s-1)}})$.
The bound recovers Theorem~4.5 of
Krivelevich and Sudakov~\cite{KrivelevichSudakov2009}.

\item
For every fixed $k\ge2$,
Lemma~\ref{lem:applications-even-cycles} gives
$\operatorname{ex}(n,C_{2k})=O_k(n^{1+1/k})$.
Thus every $C_{2k}$-free graph of average degree $d$
contains a $J$-minor with
$d(J)=\Omega_k(d^{(k+1)/2})$,
and a clique minor of order
$d^{(k+1)/2-o(1)}$.
This recovers Theorem~4.7 of Krivelevich and
Sudakov~\cite{KrivelevichSudakov2009}.

\item
Let $F$ be any fixed bipartite graph with
\textit{Tur\'an exponent}
$\gamma(F):=\limsup_{n\to\infty}
\frac{\log\operatorname{ex}(n,F)}{\log n}>1$.
Since $\operatorname{ex}(n,F)\le n^{\gamma(F)+o(1)}$,
Corollary~\ref{cor:forbidden} implies that every
$F$-free graph $G$ of sufficiently large average
degree $d$ contains a $J$-minor with
$d(J)\ge d^{\frac{\gamma(F)}
{2(\gamma(F)-1)}-o(1)}$.
This recovers Theorem~5.3 of Norin, Postle and
Song~\cite{NorinPostleSong2023}.

\end{itemize}

\subsection{Subgraphs of Hamming graphs}

Chandran and Sivadasan~\cite{ChandranSivadasan2007}
proved that the Hamming graph $H(m,q)$ contains a
clique minor of order at least $q^{\lfloor(m-1)/2\rfloor}$.
Wood~\cite{Wood2011} proved that $H(m,q)$ contains a clique minor of order at least $\lfloor q/2\rfloor\lfloor\sqrt q\rfloor^{m-1}$.
For arbitrary subgraphs, Haslegrave et
al.~\cite{HaslegraveEtAl} showed that, for fixed $q$
and sufficiently large $d$, every subgraph of
$H(m,q)$ with average degree $d$ contains a cycle
of length at least $q^{d/(q-1)-o(d)}$.
In the same setting, our result guarantees
a minor of average degree at least
$q^{d/(2(q-1))-O_q(\log d)}$
and a clique minor of the same order.
In particular, these bounds depend only on
the average degree, independently of the
ambient dimension.

\begin{corollary}\label{cor:Hamming-minors}
Fix an integer $q\ge2$ and $0<\alpha<1$.
Every subgraph $G\subseteq H(m,q)$ with average degree
$d>0$ contains a $J$-minor satisfying
\[
d(J)\ge c_\alpha\sqrt d\,
q^{\alpha d/(2(q-1))},
\]
where $c_\alpha>0$ is the constant in
Theorem~\ref{thm:crux}.
It also contains a $K_r$-minor with
$r\ge c_{\alpha,q}q^{\alpha d/(2(q-1))}$,
where $c_{\alpha,q}>0$ depends only on $\alpha$ and $q$.

Moreover, there exist constants $c_q>0$ and $d_q$
depending only on $q$ such that, whenever $d\ge d_q$,
$G$ contains a $K_s$-minor with
$s\ge c_qd^{-16}q^{d/(2(q-1))}$.
\end{corollary}

\begin{proof}
Let $F\subseteq G$ satisfy $d(F)\ge\alpha d$.
By Lemma~\ref{lem:applications-Hamming-isoperimetry},
we have $|F|\ge q^{d(F)/(q-1)}
\ge q^{\alpha d/(q-1)}$.
Hence $C_\alpha(G)\ge q^{\alpha d/(q-1)}$,
and the first assertion follows immediately from
Theorem~\ref{thm:crux}.

Let $x:=c_\alpha\sqrt d\,q^{\alpha d/(2(q-1))}$.
For all sufficiently large $d$, we have
$x\ge2$ and $\log x\le B_{\alpha,q}d$
for some constant $B_{\alpha,q}>0$.
Applying Lemma~\ref{lem:applications-KT}
to the minor $J$, we obtain a $K_r$-minor in $G$
with $r\ge \lambda_{\mathrm{KT}}x/\sqrt{\log x}
\ge c_{\alpha,q}q^{\alpha d/(2(q-1))}$.
For the remaining bounded range of $d$,
the assertion holds after decreasing $c_{\alpha,q}$,
as $d>0$ ensures that $G$ contains an edge.

For the final assertion, assume $d\ge2$ and choose
$\alpha:=1-1/d$.
The proof of Theorem~\ref{thm:crux} shows that
$c_\alpha$ can be chosen to satisfy
$c_\alpha\ge c_*(1-\alpha)^{16}=c_*d^{-16}$
for some absolute constant $c_*>0$.
Consequently, $G$ contains a minor $J$ with
\[
d(J)\ge
c_*d^{-31/2}q^{(d-1)/(2(q-1))}
=c_*q^{-1/(2(q-1))}d^{-31/2}
q^{d/(2(q-1))}.
\]
For sufficiently large $d$, the right-hand side
is at least $2$ and its logarithm is $O_q(d)$.
Applying Lemma~\ref{lem:applications-KT}
to $J$ again gives a clique minor of order at least
$c_qd^{-16}q^{d/(2(q-1))}$.
\end{proof}

We next determine the average degree threshold
for forcing a clique minor in subgraphs of
Hamming graphs.
For integers $q\ge2$ and $t\ge3$, let $F_q(t)$
be the supremum of $d(G)$ over all nonempty
$K_t$-minor-free subgraphs $G$ of $H(m,q)$,
where the dimension $m$ is unrestricted.
For general graphs, the results of
Kostochka~\cite{Kostochka} and
Thomason~\cite{Thomason} show that the
average degree threshold for forcing a
$K_t$-minor is of order $t\sqrt{\log t}$.
In contrast, the following corollary shows that
this threshold is $2(q-1)\log_q t+O_q(\log\log t)$
for subgraphs of Hamming graphs with fixed $q$.
In particular, the leading constant $2(q-1)$ is best
possible, as witnessed by the full Hamming graphs.

\begin{corollary}\label{cor:Hamming-excluded-minor}
For every fixed integer $q\ge2$, as $t\to\infty$,
\[
F_q(t)=2(q-1)\log_qt+O_q(\log\log t).
\]
In particular,
$F_2(t)=2\log_2t+O(\log\log t)$.
\end{corollary}

\begin{proof}
Fix $q$, and let $G\subseteq H(m,q)$
be $K_t$-minor-free with average degree $d$.
If $d=0$, the desired upper bound is immediate.
Hence assume $d>0$.
Corollary~\ref{cor:Hamming-minors}, applied with
$\alpha=1/2$, first gives $d=O_q(\log t)$.
For sufficiently large $d$, the final assertion
of that corollary gives
$c_qd^{-16}q^{d/(2(q-1))}<t$.
Taking logarithms, we obtain
\[
d<2(q-1)\log_qt
+32(q-1)\log_qd+O_q(1).
\]
Since $d=O_q(\log t)$, it follows that
$d\le2(q-1)\log_qt+O_q(\log\log t)$.
The same bound holds for bounded $d$, proving
the upper bound on $F_q(t)$.

For the lower bound, let $L:=\log_qt$
and choose an integer $a_q\ge1$ such that
$q^{a_q}\ge8(q-1)$.
Let $k:=\lfloor2L-\log_qL-a_q\rfloor$,
which is positive for sufficiently large $t$.
The full Hamming graph $H(k,q)$ has average degree
$(q-1)k$ and
$e(H(k,q))=(q-1)kq^k/2$.
Since $k\le2L$ and
$q^k\le t^2/(q^{a_q}L)$, we have
\[
e(H(k,q))
\le\frac{(q-1)t^2}{q^{a_q}}
\le\frac{t^2}{8}<\binom t2.
\]
As taking minors cannot increase the number of edges,
$H(k,q)$ contains no $K_t$-minor.
Therefore
$F_q(t)\ge(q-1)k
=2(q-1)\log_qt-O_q(\log\log t)$,
completing the proof.
\end{proof}

\subsection{Pseudorandom graphs}

For $0<p\le1$ and $\beta\ge0$, an $n$-vertex
graph $\Gamma$ is called \emph{$(p,\beta)$-jumbled}
if
$\left|e(\Gamma[S])-p|S|^2/2\right|
\le\beta|S|$
for every $S\subseteq V(\Gamma)$.
Krivelevich and Sudakov~\cite[Theorem~4.4]
{KrivelevichSudakov2009} proved that every
$n$-vertex $(p,\beta)$-jumbled graph with
$\beta=o(pn)$ contains a dense minor of average
degree $\Omega(n\sqrt p)$.
Our result shows that the same bound, up to a constant
depending on $\rho$, holds for every spanning subgraph
retaining at least a fixed proportion $\rho>0$ of
the original edges.
Unlike the original graph, these spanning subgraphs
need not be pseudorandom or expanding.
Moreover, the bound remains optimal up to a constant
factor, since taking minors cannot increase the
number of edges.

\begin{corollary}\label{cor:jumbled-spanning-minors}
Let $n\ge2$, $0<p\le1$, and $0<\rho\le1$.
Suppose that $\Gamma$ is an $n$-vertex
$(p,\beta)$-jumbled graph with
$\beta\le\rho pn/8$.
Then every spanning subgraph $G\subseteq\Gamma$
with $d(G)\ge\rho pn$ contains a $J$-minor satisfying
\[
d(J)\ge
\frac{c_{1/2}d(G)}{2\sqrt p}
\ge\frac{c_{1/2}\rho}{2}n\sqrt p.
\]
\end{corollary}

\begin{proof}
Write $d:=d(G)$.
Let $F\subseteq G$ satisfy $d(F)\ge d/2$,
and put $s:=|F|$.
Since $F\subseteq\Gamma[V(F)]$, we have
$d/2\le d(F)\le
2e(\Gamma[V(F)])/s\le ps+2\beta$.
As $2\beta\le\rho pn/4\le d/4$,
we obtain $s\ge d/(4p)$.
Hence $C_{1/2}(G)\ge d/(4p)$.
Applying Theorem~\ref{thm:crux} with $\alpha=1/2$,
we obtain a minor $J$ satisfying
\[
d(J)\ge c_{1/2}\sqrt{d\,C_{1/2}(G)}
\ge\frac{c_{1/2}d}{2\sqrt p}
\ge\frac{c_{1/2}\rho}{2}n\sqrt p.
\]
\end{proof}

In particular, if $\beta=o(pn)$, then for every fixed
$0<\rho\le1$, every spanning subgraph of $\Gamma$
retaining at least $\rho e(\Gamma)$ edges contains
a dense minor of average degree $\Omega_\rho(n\sqrt p)$.
%Indeed, $d(\Gamma)=(1+o(1))pn$, so such a spanning subgraph has average degree at least $\rho pn/2$ for sufficiently large $n$. The conclusion follows by applying Corollary~\ref{cor:jumbled-spanning-minors} with $\rho/2$ in place of $\rho$.

\section*{Declaration on the use of AI}
During the preparation of this work, the authors used
OpenAI's GPT-6.1 to develop the proof of Theorem~\ref{thm:crux}.
The authors had previously considered the problem
addressed by Theorem~\ref{thm:crux}
and explored some preliminary ideas for almost regular graphs.
After learning of the recent result on edge expansion
by Li, Xu and Zhou~\cite{LiXuZhou}, the authors
provided GPT-6.1 with the relevant paper and their
earlier ideas, and GPT-6.1 generated a proof of the theorem.
The authors subsequently independently verified
the mathematical arguments and associated estimates,
and revised the exposition to improve clarity
and logical coherence.
The authors take full responsibility for the
correctness and presentation of the results.

\appendix

\section{Proof of Lemma~\ref{lem:crux-cheeger}}
\label{app:crux-cheeger}

We first record three results of Li, Xu and Zhou
\cite[Corollary~2.3 and Lemmas~3.1, 3.3]{LiXuZhou}
used in the proof, with notation adapted to this paper.

\begin{lemma}[\cite{LiXuZhou}, Corollary~2.3]
\label{lem:lxz-hitting}
Let $G$ be an $n$-vertex graph with $n\ge2$,
$\Delta(G)\le D$, and $h(G)\ge\sigma D$,
where $D\ge1$ and $\sigma>0$.
Let $q\ge1$ be an integer, and $W_1,\ldots,W_q\subseteq V(G)$ satisfy
$|W_i|\ge w\ge1$ for every $i\in[q]$.
For every nonnegative integer $t$, there exists
a connected set $T\subseteq V(G)$ with
$|T|\le t+1$ which meets at least
\[
q\left(1-\exp\left(
-\frac{\sigma^3wt}{8n}\right)\right)
\]
of the sets $W_1,\ldots,W_q$.
\end{lemma}

\begin{lemma}[\cite{LiXuZhou}, Lemma~3.1]
\label{lem:lxz-core}
Let $0<\varepsilon\le1$ and $D>0$, and let
$G$ be an $n$-vertex graph with
$\Delta(G)\le D$ and $h(G)\ge\varepsilon D$.
Suppose that $Y,R_0\subseteq V(G)$ are disjoint
and satisfy $|Y|\le\varepsilon^2n/16$,
$|R_0|\le n/2$, and
\[
e_G(R_0,V(G)\setminus(Y\cup R_0))
\le\frac{\varepsilon D}{4}|R_0|.
\]
Then there is a partition
$V(G)=Y\mathbin{\dot\cup}R\mathbin{\dot\cup}U$
such that $R_0\subseteq R$ and
$|R|\le4|Y|/(3\varepsilon)$,
$|U|\ge n/2$, $h(G[U])\ge\varepsilon D/4$,
and $e_G(R,U)\le(\varepsilon D/4)|R|$.

Moreover, if $Z\subseteq U$ and
$|Y\cup Z|\le\varepsilon^2n/16$,
there is a partition
$V(G)=(Y\cup Z)\mathbin{\dot\cup}R'
\mathbin{\dot\cup}U'$
such that $R\subseteq R'$ and
$U'\subseteq U\setminus Z$ and
$|R'|\le4|Y\cup Z|/(3\varepsilon)$,
$|U'|\ge n/2$, $h(G[U'])\ge\varepsilon D/4$,
and $e_G(R',U')\le(\varepsilon D/4)|R'|$.
\end{lemma}

\begin{lemma}[\cite{LiXuZhou}, Lemma~3.3]
\label{lem:lxz-family}
For every $0<\varepsilon\le1$, there exist
constants $c=c(\varepsilon)>0, \xi=\xi(\varepsilon)>0,\zeta=\zeta(\varepsilon)>0$, and an integer $D_2=D_2(\varepsilon)$ such that $0<c\le1/32$, $0<\xi\le1/2048$,
$0<\zeta\le\varepsilon$, and the following holds.

Let $G$ be an $n$-vertex graph with
$\Delta(G)\le D$ and $h(G)\ge\varepsilon D$,
where $D\ge D_2$ is an integer.
If $\ell$ is an integer with
$2\le\ell\le n/(2D)$, then there exist
an integer $m\ge1$, pairwise disjoint
nonempty connected sets $A_1,\ldots,A_m$,
and sets $Q_1,\ldots,Q_m$ such that, writing
$X:=\bigcup_{i=1}^m A_i$, the following hold:

\begin{enumerate}[label=(\roman*)]
\item
$\xi n/\ell\le m\le4\xi n/\ell$ and
$|X|\le\varepsilon\zeta n/32$;

\item
$|A_i|\le\ell$,
$Q_i\subseteq N_G(A_i)\setminus X$, and
$|Q_i|\ge cD\ell$ for every $i\in[m]$;

\item
For every $Z\subseteq V(G)$ with
$|Z|\le\zeta n$, fewer than $m/16$
indices $i\in[m]$ satisfy
$|Q_i\setminus Z|<cD\ell$.
\end{enumerate}
\end{lemma}

The next crucial tool (Lemma~\ref{lem:crux-connected-hitting}) is a quantitative variant of the connected hitting set lemma of Krivelevich and Nenadov~\cite[Lemma~3.1]{KrivelevichNenadov}. Their result assumes $qw\ge2n$ and $n$ sufficiently large, whereas we require $qw\le n$ and an explicit bound of $16\sigma^{-3}n/w$ valid for every $n\ge2$. Although the difference in the range of $qw$ can be handled by repeating sets, their result does not directly yield the required uniform bound. Our proof follows essentially the same random walk argument, with more detailed estimates to track the constants explicitly and handle small $n$, using their random walk estimate~\cite[Lemma~2.4]{KrivelevichNenadov}.

\begin{lemma}[\cite{KrivelevichNenadov}, Lemma~2.4]
\label{lem:KN-random-walk}
Let $G$ be an $n$-vertex graph with $n\ge2$
and maximum degree $\Delta(G)\ge1$.
For every $W\subseteq V(G)$ and every
positive integer $t$, a stationary lazy
random walk $(X_0,\ldots,X_t)$ in $G$
satisfies
\[
\Prob(X_i\notin W\text{ for all }0\le i\le t)
\le
\exp\left(
-\frac{h(G)^3|W|t}{8\Delta(G)^3n}
\right).
\]
\end{lemma}

\begin{lemma}\label{lem:crux-connected-hitting}
Let $G$ be an $n$-vertex graph with $n\ge2$,
$\Delta(G)\le D$, and $h(G)\ge\sigma D$,
where $D\ge1$ and $0<\sigma\le1$.
Let $q\ge1$ be an integer, and $W_1,\ldots,W_q\subseteq V(G)$ satisfy $|W_i|\ge w\ge1$ for every $i\in[q]$,
where $w$ is real and $qw\le n$.
Then there exists a nonempty connected set
$T\subseteq V(G)$ meeting every $W_i$ such that
$|T|\le16\sigma^{-3}n/w$.
\end{lemma}

\begin{proof}
Since $h(G)\ge\sigma D$, the graph $G$ is connected and $\delta(G)\ge\sigma D$.
Hence the stationary distribution of the lazy random walk satisfies $\pi(v)=d_G(v)/(2e(G))\le1/(\sigma n)$ for every $v\in V(G)$.
Let $t:=\lceil8n/(\sigma^3w)\rceil$ and $T_0$ be the set of vertices visited
by a stationary lazy random walk of length $t$
in $G$. Then $G[T_0]$ is connected and
$|T_0|\le t+1$.
By Lemma~\ref{lem:KN-random-walk}, using $h(G)/\Delta(G)\ge\sigma$, we have
$\Prob(T_0\cap A=\varnothing)
\le\exp(-\sigma^3|A|t/(8n))$
for every $A\subseteq V(G)$.

Fix $i\in[q]$ and let
$X_i:=\dist_G(T_0,W_i)$.
For $j\ge0$, let
$B_j:=B_G^{(j)}(W_i)
=\{v\in V(G):\dist_G(v,W_i)\le j\}$,
and $j_0$ be the smallest integer
such that $|B_{j_0}|\ge n/2$.
For every $B\subseteq V(G)$ with
$0<|B|\le n/2$, we have $|N_G(B)|\ge\sigma|B|$.
Consequently, $|B_j|\ge w(1+\sigma)^j$
for every $0\le j<j_0$.
By Lemma~\ref{lem:KN-random-walk},
$\Prob(X_i>j)
=\Prob(T_0\cap B_j=\varnothing)
\le\exp(-(1+\sigma)^j)$.
Since $(1+\sigma)^j\ge1+j\sigma$ and
$1-e^{-\sigma}\ge\sigma/2$, we obtain
$\sum_{0\le j<j_0}\Prob(X_i>j)
\le e^{-1}/(1-e^{-\sigma})\le1/\sigma$.

For $j\ge j_0$, let
$C_j:=V(G)\setminus B_j$.
Since $|C_j|\le n/2$, we have $e_G(C_j,B_j)\ge\sigma D|C_j|$. As $\Delta(G)\le D$, at least
$\sigma|C_j|$ vertices of $C_j$ have
a neighbour in $B_j$.
Thus $|C_{j+1}|\le(1-\sigma)|C_j|$,
and hence
$|C_{j_0+k}|\le(n/2)(1-\sigma)^k$
for every $k\ge0$.
If $T_0$ avoids $B_j$, then the starting
vertex of the walk belongs to $C_j$.
Therefore
$\Prob(X_i>j)\le\pi(C_j)
=\sum_{v\in C_j}\pi(v)
\le|C_j|/(\sigma n)$, which implies $\sum_{j\ge j_0}\Prob(X_i>j) \le1/(2\sigma^2)$.

Combining the two ranges yields
$\E X_i\le1/\sigma+1/(2\sigma^2)
\le2/\sigma^2$.
For each $i\in[q]$, choose a shortest path
from $T_0$ to $W_i$, and let $T$ be the union
of $T_0$ and all these paths.
Then $T$ is connected, meets every $W_i$,
and satisfies
$|T|\le t+1+\sum_{i=1}^qX_i$.
By linearity of expectation, $qw\le n$,
$\sigma\le1$, and $n/w\ge1$, we have
\[
\E|T|
\le\frac{8n}{\sigma^3w}+2+\frac{2q}{\sigma^2}
\le12\sigma^{-3}\frac nw
\le16\sigma^{-3}\frac nw.
\]
Hence there exists a choice of the random walk for which the resulting connected set $T$ meets every $W_i$ and satisfies $|T|\le16\sigma^{-3}n/w$, completing the proof.
\end{proof}

The following lemma gives a quantitative version of the clique minor bound of Krivelevich and Nenadov~\cite[Theorem~1.2]{KrivelevichNenadov}, with explicit cubic dependence on $\varepsilon$ and valid for every $n\ge2$. The proof follows essentially the same maximal partition argument as theirs. We include the details to track the dependence on $\varepsilon$ explicitly, using Lemma~\ref{lem:crux-connected-hitting}, and to handle the small order case.
\begin{lemma}\label{lem:crux-quantitative-clique}
Let $G$ be an $n$-vertex graph with $n\ge2$,
$\Delta(G)\le D$, and $h(G)\ge\varepsilon D$,
where $D\ge1$ and $0<\varepsilon\le1$.
Then $G$ contains a $K_t$-minor for some
integer $t\ge2$ such that
$t-1\ge2^{-15}\varepsilon^3\sqrt n$.
\end{lemma}

\begin{proof}
Let $m:=\lceil64\varepsilon^{-2}\sqrt n\rceil$,
$x:=\varepsilon n/(64m)$, and $r:=\lfloor x\rfloor$.
We first show that $G$ contains a $K_r$-minor
whenever $r\ge2$. Consider all partitions
\[
V(G)=R\mathbin{\dot\cup}T_1
\mathbin{\dot\cup}\cdots
\mathbin{\dot\cup}T_q\mathbin{\dot\cup}U,
\]
where $q\ge0$, each $T_i$ is a connected
set of exactly $m$ vertices, every two
distinct sets $T_i,T_j$ are adjacent,
and $|R|\le2n/3$ and
$e_G(R,U)\le(\varepsilon D/4)|R|$.
Such a partition exists by taking
$R=\varnothing$, $q=0$, and $U=V(G)$.
Choose one maximizing $|R|$, and subject
to this, maximizing $q$.

If $q\ge r$, then the sets $T_i$ give
the required minor. Suppose that $q<r$,
and let $B:=\bigcup_{i=1}^qT_i$.
Then $|B|=qm<rm\le\varepsilon n/64$.
Since $|R|\le2n/3$, we have
$e_G(R,V(G)\setminus R)
\ge\varepsilon D\min\{|R|,n-|R|\}
\ge(\varepsilon D/2)|R|$, which implies $(\varepsilon D/4)|R|
\le e_G(R,B)\le D|B|$.
Consequently,
$|R|\le4|B|/\varepsilon<n/16$
and $|U|>59n/64$.
Moreover, since $r\ge2$ and
$rm\le\varepsilon n/64$, we have
$m\le\varepsilon n/128\le n/128$.

We claim that $h(G[U])\ge\varepsilon D/4$.
Otherwise, there exists a set
$S\subseteq U$ with $0<|S|\le|U|/2$
such that
$e_G(S,U\setminus S)
<(\varepsilon D/4)|S|$.
Moving $S$ from $U$ into $R$ preserves
the cut condition, since
$e_G(R\cup S,U\setminus S)
\le e_G(R,U)+e_G(S,U\setminus S)
\le(\varepsilon D/4)|R\cup S|$.
Furthermore,
$|R\cup S|\le|R|+|U|/2
=(n+|R|-|B|)/2
<17n/32<2n/3$.
This contradicts the maximality of $|R|$.

We next claim that
$|N_G(T_i)\cap U|>\varepsilon m/4$
for every $i\in[q]$.
Indeed, if
$e_G(T_i,U)\le(\varepsilon D/4)m$,
we could move $T_i$ into $R$ and remove
it from the family of branch sets.
The cut condition would still hold,
and $|R\cup T_i|
<n/16+n/128<2n/3$,
contradicting the maximality of $|R|$.
Thus $e_G(T_i,U)>(\varepsilon D/4)m$,
and the claim follows from
$\Delta(G)\le D$.

Suppose $q\ge1$, and let
$W_i:=N_G(T_i)\cap U$ for $i\in[q]$
and $w:=\varepsilon m/4$.
Then $|W_i|>w\ge16\varepsilon^{-1}\sqrt n
\ge1$, while
$qw=\varepsilon|B|/4
<\varepsilon^2n/256<|U|$.
Since $\Delta(G[U])\le D$ and
$h(G[U])\ge(\varepsilon/4)D$,
we may apply
Lemma~\ref{lem:crux-connected-hitting}
to $G[U]$ with $\sigma=\varepsilon/4$
and $w=\varepsilon m/4$.
This gives a connected set $T\subseteq U$
meeting every $W_i$ such that
\[
|T|\le16\left(\frac4\varepsilon\right)^3
\frac{|U|}{w}
\le\frac{4096\varepsilon^{-4}n}{m}
\le m,
\]
where the last inequality follows from
$m\ge64\varepsilon^{-2}\sqrt n$.
Since $G[U]$ is connected and $|U|\ge m$,
we can enlarge $T$ inside $U$ to a
connected set of exactly $m$ vertices.
This set is adjacent to every $T_i$.
If $q=0$, we instead choose any connected
set of $m$ vertices in $G[U]$.

In either case, denote the resulting set
by $T_{q+1}$ and remove it from $U$.
The set $R$ remains unchanged, and
its cut to the remaining vertices of $U$
cannot increase.
Thus the new partition satisfies all
the required conditions and has a larger
value of $q$, a contradiction.
This proves that $G$ contains a
$K_r$-minor whenever $r\ge2$.

Finally, $m\le128\varepsilon^{-2}\sqrt n$
implies $x\ge2^{-13}\varepsilon^3\sqrt n$.
If $x\ge2$, then $r\ge2$ and
$r-1\ge r/2\ge x/4
\ge2^{-15}\varepsilon^3\sqrt n$.
Taking $t=r$ gives the desired bound.
If $x<2$, then
$2^{-15}\varepsilon^3\sqrt n
\le x/4<1/2$.
Since $h(G)>0$ and $n\ge2$,
the graph $G$ contains an edge.
Taking $t=2$ completes the proof.
\end{proof}

\begin{proof}[Proof of Lemma~\ref{lem:crux-cheeger}]
Given $0<\varepsilon\le1$ and an $n$-vertex graph $G$ with $n\ge2$, $\Delta(G)\le D$ and $h(G)\ge\varepsilon D$,
we choose constants \[\kappa:=2^{-20},~ p:=\kappa\varepsilon^2,
~c:=\varepsilon/32,~\xi:=p\varepsilon^2/512, ~\text{and}~
\zeta:=p\varepsilon/4.\]
These are precisely the parameters used in the
proof of Lemma~\ref{lem:lxz-family}
in~\cite{LiXuZhou}.
Moreover, the degree threshold $D_2$ in that proof
can be taken to be at most
$\lceil2^{29}\varepsilon^{-4}\rceil$.
Let \[\sigma:=\varepsilon/4,~
C:=1+8\log4/(\sigma^3c),~ \text{and}~
\theta:=1/(100C).\]
Then $C=O(\varepsilon^{-4})$ and
$\theta\xi=\Omega(\varepsilon^8)$.
Fix a sufficiently large absolute constant
$K$ such that $D\ge K\varepsilon^{-8}$
implies $D\ge\lceil2^{29}\varepsilon^{-4}\rceil$,
$cD\ell\ge1$ for every $\ell\ge2$,
and $\theta\xi D\ge1$.

Suppose that $D\ge K\varepsilon^{-8}$
and $n\ge16D$.
Choose $\ell:=\lceil\sqrt{n/D}\rceil$.
Then we have $2\le\ell\le n/(2D)$,
$D\ell^2\ge n$, and $\ell\le2\sqrt{n/D}$.
By Lemma~\ref{lem:lxz-family}, there are
pairwise disjoint connected sets
$A_1,\ldots,A_m$ and sets
$Q_1,\ldots,Q_m$ such that, writing
$X:=\bigcup_{i=1}^m A_i$, we have
$\xi n/\ell\le m\le4\xi n/\ell$,
$|X|\le\varepsilon\zeta n/32$,
and $Q_i\subseteq N_G(A_i)\setminus X$
with $|Q_i|\ge cD\ell$ for every $i\in[m]$.
Moreover, for every $Z\subseteq V(G)$
with $|Z|\le\zeta n$, at least $15m/16$
indices $i$ satisfy
$|Q_i\setminus Z|\ge cD\ell$.

Let $k:=\lfloor\theta m\rfloor$.
Since $m\ge\xi n/\ell\ge2\xi D$ and
$\theta\xi D\ge1$, we have
$\theta m/2\le k\le m/16$.
We construct connected sets $T_1,\ldots,T_k$
inductively, together with auxiliary partitions.
Let $S_0:=\varnothing$ and
$S_j:=\bigcup_{r=1}^j T_r$ for $j\in[k]$.
We claim that, for each $j\in[k]$,
there exist pairwise disjoint connected sets
$T_1,\ldots,T_j$, all disjoint from $X$,
and partitions $\mathcal P_r:
V(G)=(X\cup S_{r-1})
\mathbin{\dot\cup}B_r
\mathbin{\dot\cup}U_r, r\in[j]$,
such that, for every $r\in[j]$,
$T_r\subseteq U_r$, $|T_r|\le C\ell$,
and $T_r$ meets at least $45m/64$
of the sets $Q_i$.
Moreover, $|U_r|\ge n/2$,
$h(G[U_r])\ge\sigma D$,
$e_G(B_r,U_r)\le\sigma D|B_r|$,
$|B_r|\le4|X\cup S_{r-1}|/(3\varepsilon)$,
and $|B_r\cup S_{r-1}|<\zeta n$.

We first prove the claim for $j=1$.
Since $|X|\le\varepsilon\zeta n/32
\le\varepsilon^2n/16$,
applying Lemma~\ref{lem:lxz-core}
with $(Y,R_0)=(X,\varnothing)$
gives a partition
$\mathcal P_1$ satisfying
$|U_1|\ge n/2$,
$h(G[U_1])\ge\sigma D$,
$e_G(B_1,U_1)\le\sigma D|B_1|$,
and
$|B_1|\le4|X|/(3\varepsilon)
\le\zeta n/24<\zeta n$.

We now construct $T_1$ using this partition.
Since $Q_i\cap X=\varnothing$,
we have $Q_i\cap U_1=Q_i\setminus B_1$.
By Lemma~\ref{lem:lxz-family}(iii),
at least $15m/16$ indices $i$ satisfy
$|Q_i\cap U_1|\ge cD\ell$.
Let $i_1,\ldots,i_{q_1}$ be these indices,
where $q_1\ge15m/16$.
Applying Lemma~\ref{lem:lxz-hitting}
to $G[U_1]$ with
$(n,D,\sigma,q,w)
=(|U_1|,D,\sigma,q_1,cD\ell)$,
$W_s=Q_{i_s}\cap U_1$ for $s\in[q_1]$,
and
$t:=\lceil8|U_1|\log4/(\sigma^3cD\ell)\rceil$,
we obtain a connected set $T_1\subseteq U_1$
meeting at least $3q_1/4\ge45m/64$
of the sets $Q_i$.
Moreover,
$|T_1|\le t+1\le Cn/(D\ell)\le C\ell$.
Since $U_1\cap X=\varnothing$,
the set $T_1$ is disjoint from $X$.
This proves the claim for $j=1$.

Now suppose that the claim holds for some
$1\le j<k$.
We first construct the partition
$\mathcal P_{j+1}$.
Since each $T_r$ has size at most $C\ell$,
we have
\[
|S_j|\le jC\ell\le kC\ell
\le4\theta\xi Cn
=\frac{\xi n}{25}
\le\frac{\varepsilon\zeta n}{32}.
\]
Consequently,
$|X\cup S_j|\le\varepsilon\zeta n/16
\le\varepsilon^2n/16$.
Since $T_j\subseteq U_j$,
we may apply the second part of
Lemma~\ref{lem:lxz-core} with
$(Y,R,U,Z)
=(X\cup S_{j-1},B_j,U_j,T_j)$.
Writing $(R',U')=(B_{j+1},U_{j+1})$,
we obtain a partition
$\mathcal P_{j+1}$ such that
$B_j\subseteq B_{j+1}$,
$U_{j+1}\subseteq U_j\setminus T_j$,
$|U_{j+1}|\ge n/2$,
$h(G[U_{j+1}])\ge\sigma D$,
$e_G(B_{j+1},U_{j+1})
\le\sigma D|B_{j+1}|$, and
$|B_{j+1}|
\le4|X\cup S_j|/(3\varepsilon)$.
Furthermore,
$|B_{j+1}\cup S_j|
\le\zeta n/12+\varepsilon\zeta n/32
<\zeta n$.
Thus $\mathcal P_{j+1}$ satisfies
all the required properties.

We next use $\mathcal P_{j+1}$
to construct $T_{j+1}$.
Since $Q_i\cap X=\varnothing$,
we have
$Q_i\cap U_{j+1}
=Q_i\setminus(B_{j+1}\cup S_j)$.
By Lemma~\ref{lem:lxz-family}(iii),
at least $15m/16$ indices $i$ satisfy
$|Q_i\cap U_{j+1}|\ge cD\ell$.
Let $i_1,\ldots,i_{q_{j+1}}$ be these
indices, where $q_{j+1}\ge15m/16$.
Applying Lemma~\ref{lem:lxz-hitting}
to $G[U_{j+1}]$ with
$(n,D,\sigma,q,w)
=(|U_{j+1}|,D,\sigma,q_{j+1},cD\ell)$,
$W_s=Q_{i_s}\cap U_{j+1}$
for $s\in[q_{j+1}]$, and
$t:=\lceil8|U_{j+1}|\log4/
(\sigma^3cD\ell)\rceil$,
we obtain a connected set
$T_{j+1}\subseteq U_{j+1}$
meeting at least
$3q_{j+1}/4\ge45m/64$
of the sets $Q_i$.
Moreover,
$|T_{j+1}|\le t+1
\le Cn/(D\ell)\le C\ell$.
Since $T_{j+1}\subseteq U_{j+1}$
and $U_{j+1}$ is disjoint from
$X\cup S_j$, the sets
$T_1,\ldots,T_{j+1}$ are pairwise
disjoint and disjoint from $X$.
This completes the induction
and yields the desired sets
$T_1,\ldots,T_k$.

Contract the sets $A_1,\ldots,A_m$
and $T_1,\ldots,T_k$ separately,
retaining only edges between the
two families.
Since each $T_j$ meets at least
$45m/64$ of the sets $Q_i$ and
$Q_i\subseteq N_G(A_i)$,
the resulting $J$-minor satisfies $e(J)\ge45mk/64$.
As $k\le m$, $k\ge\theta m/2$,
$m\ge\xi n/\ell$, $\ell\le2\sqrt{n/D}$,
and $\theta\xi=\Omega(\varepsilon^8)$, we obtain
\[
d(J)=\frac{2 e(J)}{m+k}\ge\frac{45}{64}k
\ge\frac{45\theta\xi}{128}\frac{n}{\ell}
\ge c_1\varepsilon^8\sqrt{nD}
\]
for some absolute constant $c_1>0$.

If $D\ge K\varepsilon^{-8}$ but
$n<16D$, then
$\delta(G)\ge h(G)\ge\varepsilon D$.
Taking $J:=G$ gives
$d(J)\ge(\varepsilon/4)\sqrt{nD}
\ge(\varepsilon^8/4)\sqrt{nD}$.

Finally, suppose $D<K\varepsilon^{-8}$.
By Lemma~\ref{lem:crux-quantitative-clique},
$G$ contains a $K_t$-minor with
$t-1\ge2^{-15}\varepsilon^3\sqrt n$.
Taking $J:=K_t$ and using
$\sqrt D<\sqrt K\,\varepsilon^{-4}$
and $0<\varepsilon\le1$, we obtain
$d(J)=t-1
\ge2^{-15}\varepsilon^3\sqrt n
\ge(2^{-15}/\sqrt K)\varepsilon^7\sqrt{nD}
\ge(2^{-15}/\sqrt K)\varepsilon^8\sqrt{nD}$.

Thus, in all three cases, $G$ contains
a $J$-minor satisfying
$d(J)\ge a\varepsilon^8\sqrt{nD}$,
where $a:=\min\{c_1,1/4,2^{-15}/\sqrt K\}>0$ is an absolute constant.
\end{proof}

\end{document}